\documentclass[a4paper,12pt]{amsart}

\usepackage{amsmath, amssymb, amsthm}
\usepackage{url}

\newcommand{\sP}{\mathsf{P}}
\newcommand{\bP}{\mathbb{P}}

\newcommand{\cS}{\mathcal{S}}

\newcommand{\sH}{\mathsf{H}}

\numberwithin{equation}{section}
\theoremstyle{plain}
\newtheorem{theorem}{Theorem}[section]
\newtheorem{lemma}[theorem]{Lemma}
\newtheorem{proposition}[theorem]{Proposition}

\newtheorem{coro}{Corollary}

\theoremstyle{definition}
\newtheorem{definition}[theorem]{Definition}

\theoremstyle{remark}
\newtheorem{remark}[theorem]{Remark}

\begin{document}
\title[]{An extremal subharmonic function  
in non-archimedean potential theory}
\author[M. Stawiska]{Ma{\l}gorzata Stawiska}
\address{Mathematical Reviews, 416 Fourth St., Ann Arbor, MI 48103, USA}
\email{stawiska@umich.edu}

\date{\today}


\begin{abstract}
 We define an  analog of the Leja-Siciak-Zaharjuta subharmonic extremal function for a proper   subset $E$ of the Berkovich projective line $\sP^1$ over a field with a non-archimedean absolute value, relative to a point $\zeta \not \in E$. When $E$ is a compact set with positive capacity we  prove that the upper semicontinuous regularization of this extremal function equals  the Green function of $E$ relative to $\zeta$. As a separate result, we prove the Brelot-Cartan principle, under the additional assumption  that the Berkovich topology is second countable.
\end{abstract}

\maketitle

\section{Introduction}\label{sec:intro}

Potential theory on curves over a field $K$
 complete
with respect to a  non-archimedean absolute value $|\cdot|$ has made substantial progress  in recent years. Fundamental developments can be found in \cite{Rumely89}, \cite{FJbook}, \cite{BR10} and \cite{ThuillierThesis}. Nevertheless,  some topics have been left unexplored, for instance extremal subharmonic functions associated with compact subsets. Such functions  are well known in classical potential and pluripotential theory in  $\mathbb{C}^N, \ N \geq 1$.  Recall that in $\mathbb{C}$ (with the standard, archimedean absolute value) the Green function with pole at infinity of a compact subset $E$ can be proved to be equal to the so-called Leja extremal function (see \cite{Leja34}, \cite{Leja45}, \cite{Gorski48}). In $\mathbb{C}^N, \ N>1$, an analogous extremal function (Siciak-Zaharjuta extremal function, \cite{Siciak62}, \cite{Za77}, \cite{Siciak81}), with plurisubharmonic functions replacing subharmonic ones, serves a multivariate counterpart to  the  Green function. In this article we will work on the Berkovich projective line $\sP^1$ over an algebraically closed complete $K$. We will fix a point  $\zeta \in \sP^1$ (not necessarily equal to $\infty$) and define a class of subharmonic functions  on $\sP^1 \setminus \{\zeta\}$ with suitable behavior near $\zeta$. The supremum $Q_E$ of this class can be treated  as a non-archimedean analogue of the Leja-Siciak-Zaharjuta extremal function. We will further show that the Green function relative to a point $\zeta$ of a  compact subset of $\sP^1 \setminus \{\zeta\}$ of positive capacity (see subsection \ref{sec: Green} for definitions) equals the upper semicontinuous regularization $Q_E^*$ of $Q_E$.  Our approach is analytic and topological rather than  geometric. In particular, we do not appeal to available results in the (already rich) non-archimedean pluripotential theory. \\

The paper is organized as follows: Throughout,  we provide proofs only for statements that are new in the non-archimedean setting.   When we explicilty state  known results (without proofs), we do so for convenient reference. In Section \ref{sec: foundations} we gather the necessary background in  potential theory on the Berkovich projective line, although we do not always present all details. So this section is mainly a  survey, but a few results therein  are  new in the non-archimedean setting (or at least not explicitly stated in the existing literature). Our first major result is  the Brelot-Cartan principle (Theorem  \ref{prop: Cartanprinciple}), proved in Section \ref{sec: Cartan}. In the proof of this principle we use the Choquet topological lemma, which  requires the underlying topological space to be Hausdorff and second countable. So  we assume  the existence of a countable base of open sets for the Berkovich topology on $\sP^1$, which is anyway compact, and hence Hausdorff. The second countability assumption is only made in this section  and  is not necessary for other results in our paper (which do not rely on the non-archimedean Brelot-Cartan principle, either). In Section \ref{sec: extremal} we fix a point  $\zeta \in \sP^1$ (not necessarily equal to $\infty$) and, for a  set $E$ such that $\overline{E} \subset  \sP^1 \setminus \{\zeta\}$, we define a suitable  class of subharmonic functions  on $\sP^1 \setminus \{\zeta\}$ (an analog of the Lelong class)  and then study properties of its supremum $Q_E$.  This function is a non-archimedean analogue of the Leja-Siciak-Zaharjuta extremal function;  we prove several results that justify this analogy. \\

 Our  main result is the equality between the upper semicontinuous regularization $Q_E^*$ of $Q_E$ and the Green function for $E$ relative to $\zeta$ when $E$ has positive capacity. In Subsection \ref{sec: classical}, we compare the non-archimedean case with the (one-dimensional, complex) classical case. This comparison also has  a survey character, with some historical notes. As a new insight, we point out  that a proof of the equality between the upper semicontinuous regularization of the Siciak-Zaharjuta function and the  Green function can be obtained in the same way as in our non-archimedean proof, instead of relying on separate identities of these two functions with  the Leja extremal function.

\section{Foundations of potential theory on the Berkovich projective line} \label{sec: foundations}

This section  recalls  background notions and results developed  by other researchers in  potential theory on the Berkovich projective line (following mostly \cite{BR10}), and so it can be skipped by readers familiar with the material. However, Proposition \ref{prop: harmlocunifbound} is a new result in the Berkovich setting. 

\subsection{Berkovich projective line as  a topological space}

Let $K$  be an algebraically closed field 
(possibly of characteristic $>0$) that is complete
with respect to a non-trivial and non-archimedean absolute value $|\cdot|$.
The Berkovich projective line $\sP^1=\sP^1(K)$ is 
the Berkovich analytification of the (classical) projective line 
$\bP^1=\bP^1(K)=K\cup\{\infty\}$. Each point in $\sP^1$ corresponds to an equivalence class of multiplicative seminorms on the polynomial ring $K[X,Y]$ extending the absolute value $|\cdot|$. 
The Berkovich upper half space
is $\sH^1=\sH^1(K):=\sP^1\setminus\bP^1$. Taking into account  the correspondences between points, seminorms and sequences of disks, the points in $\sP^1$ can be further classified into four types, with $\bP^1$ being the set of all points of type I. For some fields $K$ the set of points of type IV in $\sP^1(K)$  may be empty. For further details see Chapter 2 of \cite{BR10}, or \cite{Jonsson15}. \\

 We will work only with the Berkovich topology on the Berkovich projective line (also called the weak topology). A basis for this topology  is given in Proposition 2.7  in \cite{BR10}, while  neighborhood bases for the four types  of points are  described in the discussion after Lemma 2.28 in that book. It follows that  $\sP^1$ with the Berkovich topology is locally connected. In particular, the  connected components of open sets are open. We will use this property several times in this paper. Also, $\sP^1$ with the Berkovich topology is  uniquely arcwise connected (\cite{BR10}, Lemma 2.10).\\

Potential theory  on the Berkovich projective line makes extensive use of the tree structure on  $\sP^1$. We will not present here  the definition or general properties of this tree structure, even though it enters the formulation of some definitions and results we need. For details, we refer the reader to \cite{BR10} or \cite{Jonsson15}. 

\begin{definition}
(i) (cf. p. 39, \cite{BR10}) Let $\{\zeta_1,...,\zeta_n\} \subset \sH^1$ be  a finite set of points. The finite subgraph with endpoints $\{\zeta_1,...,\zeta_n\}$  is the intersection of all subtrees of $\sP^1$  containing the set  $\{\zeta_1,...,\zeta_n\}$.\\
(ii) (cf. Definition 2.27, \cite{BR10}) A simple domain is a domain $U \subset \sP^1$ such that $\partial U$ is a nonempty finite set $\{\zeta_1,...,\zeta_n\} \subset \sH^1$, where each $\zeta_i$ is of type II or III.
\end{definition}

\begin{proposition} (Corollary 7.11, \cite{BR10}): \label{prop: exhaustion} If $U \subsetneq \sP^1$ is a domain, then $U$ can be exhausted by a sequence $V_1 \subset V_2\subset ...$ of strict simple domains (i.e., simple domains  with  boundary points all of type II) such that $\overline{V_n} \subset V_{n+1}$ for all $n$.
\end{proposition}

Although the Berkovich topology on $\sP^1$ is not metrizable in general, an analog of the chordal metric plays an important role in the potential theory. Namely, 
the spherical kernel, or the  Hsia kernel $[x,y]_g$ on $\sP^1$ with respect to the Gauss point $g$ (a distinguished point in the Berkovich unit disk) 
is the unique upper semicontinuous and separately continuous extension to $\sP^1\times\sP^1$
of the chordal metric 
\[
[z,w] = \frac{|x_1y_2-x_2y_1|}{\max\{|x_1|,|y_1|\} \max\{|x_2|,|y_2|\}}
\] 
defined for $z=(x_1 :y_1), w=(x_2:y_2)$ in $\bP^1\times\bP^1$. More generally, for each $\zeta\in\sP^1$,
the generalized Hsia kernel on $\sP^1$ with respect to $\zeta$ is defined as
\begin{gather*}
[x,y]_{\zeta}:=[x,y]_g/([x,\zeta]_g[y,\zeta]_g)\quad\text{on }\sP^1\times\sP^1. 
\end{gather*}

Note that the kernel function 
$[x,y]_{\zeta}$ satisfies the strong triangle inequality
\begin{gather*}
 [x,y]_{\zeta}\le\max\bigl\{[x,z]_{\zeta},[z,y]_{\zeta}\bigr\}
\quad\text{for any }x,y,z\in \sP^1.
\end{gather*}

Let us fix a $\zeta\in\sP^1$. The generalized spherical kernel  $[\cdot,\cdot]= [\cdot,\cdot]_{\zeta}$ has the following properties:  (cf. \S 4.3 in \cite{BR10}): (i) $0 \leq [x,y] \leq 1$ (Formula 4.19)\\
(ii) (Proposition 4.7 (A), \cite{BR10}) $[\cdot,\cdot]$ is continuous on the complement of the diagonal and at every $(x,x) \in \mathbb{P}^1\times \mathbb{P}^1$;\\
(iii) (Proposition 4.7 (D), \cite{BR10}) For each $a \in \sP^1$ and $r \in \mathbb{R}$, the closed ball $\mathcal{B}(a,r):=\{x \in \sP^1: [x,a]\leq r\}$ is connected and closed in the Berkovich topology. It is empty if $r < [a,a]$ and coincides with   $\mathcal{B}(b,r)$ for some $b \in \mathbb{P}^1$ if $r >[a,a] =: {\rm diam }(a)$ or if $r= {\rm diam }(a)$ and $a$ is of type II or III. If $r= {\rm diam }(a)$ and $a$ is of type I or IV, then $\mathcal{B}(a,r)=\{a\}$.\\

Finally, for an arbitrary function $f: \sP^1 \to \mathbb{R}$ we define its upper semicontinuous regularization as  $f^*(z):= \limsup_{y \to z} f(y)$.

\subsection{Subharmonic functions}\label{sec: sh} In this subsection we    recall the basics on  harmonic and subharmonic functions  on the Berkovich projective line. We need to introduce the Laplacian on $\sP^1$,  but the available definitions  are all quite involved, so we will only outline the theory, mostly following  \cite{BR10}. \\

Let $\Gamma$ be a finite subgraph in $\sP^1$ (viewed as a metric graph with the path distance $\rho$ on $\sP^1$; see Section 2.7 in \cite{BR10} for the definition of $\rho$).

\begin{definition} (see Section 3.2 in \cite{BR10})
 A function $f: \Gamma \to \mathbb{R}$ is piecewise affine on $\Gamma$ if there is a set $S_f \subset \Gamma$  such that (i) $\Gamma \setminus S_f$ is a union of intervals, each of which has two distinct endpoints in $\Gamma$ and (ii) $f$ is affine on each interval in $\Gamma \setminus S_f$ with respect to its arclength parametrization. By $CPA(\Gamma)$ we denote the class of continuous, piecewise affine real-valued functions on $\Gamma$.
\end{definition}

For  a point $x \in \sP^1$ the set $\sP^1 \setminus \{x\}$ does not have to be connected. The connected components of $\sP^1 \setminus \{x\}$ can be identified with the tangent directions $\vec{v} \in T_x$ at $x$, defined as certain equivalence classes of paths emanating from $x$ (see Appendix B of \cite{BR10} and the end of Section 3.1 in \cite{BR10}).  For any function $f \in CPA(\Gamma)$, any $p \in \Gamma$ and any tangent direction $\vec{v}$ to $\Gamma$ at $p$, the directional derivative $d_{\vec{v}}f(p):=\lim_{t \to 0^+}\frac{f(\gamma(t))-f(p))}{t}$, where $\gamma$ is a representative path emanating from $p$, exists and is finite. \\

The Laplacian $\Delta$ is first defined for $f \in CPA(\Gamma)$, as 

\[
\Delta(f):=-\sum_{p \in \Gamma}\biggl (\sum_{\vec{v}\in T_p(\Gamma)} d_{\vec{v}}f(p) \biggr )\delta_p,
\]
where $\delta_p$ denotes the Dirac  measure at $p$.  Then the class of functions  $BDV (\Gamma)$ is defined (see Section 3.5 in \cite{BR10}) and the  Laplacian $\Delta_\Gamma(f)$ is defined for $f \in BDV (\Gamma)$ (Theorem 3.6 in \cite{BR10}) extending the  operator $\Delta$. That is, $\Delta_\Gamma(f) =\Delta(f)$ when $f \in CPA(\Gamma) (\subset BDV (\Gamma)).$\\

Let $U \subset \sP^1$ be a domain. Using the natural retraction maps $r_{\overline{U},\Gamma}$ for each finite subgraph $\Gamma \subset U$ (see Section 2.5 in \cite{BR10}), the  Laplacian is extended to $\Delta_{\overline{U}}(f)$ for functions $f$ in the class $BDV(U)$ (for the precise definitions, see  Definitions 5.11 and 5.15 in \cite{BR10}). Each such $\Delta_{\overline{U}}(f)$ is a finite signed Borel measure on $\overline{U}$. One further defines $\Delta_U(f):=\Delta_{\overline{U}}(f)|_U$ and $\Delta_{\partial U}(f):=\Delta_{\overline{U}}(f)|_{\partial U}$ for $f \in BDV(U)$. \\

Example (Example 5.19, \cite{BR10}): Fix $\zeta \in \sP^1$ and $y \neq \zeta$. Let $f(x):=-\log [x,y]_\zeta$. Then $f \in BDV(\sP^1)$ and $\Delta_{\sP^1}(f)=\delta_y(x)-\delta_\zeta(x)$.\\

Let $U$ be an open set in $\sP^1$.
\begin{definition} (Definition 8.1, \cite{BR10}) (i) A function $f: U \to [-\infty,\infty)$ is  subharmonic in $U$ if for each $x \in U$ there is a domain $V_x \subset U$ with $x \in V_x$ such that $f \in BDV(V_x)$, $\Delta_{V_x}(f) \leq 0$, $f$ is upper semicontinuous in $V_x$, and for each $z \in V_x \cap \mathbb{P}^1$, $f(z)=\limsup_{V_x \cap \sH^1 \ni y \to z}f(y)$.\\
(ii)  A function $f: U \to \mathbb{R}$ is harmonic in $U$ if and only if $f$ and $-f$ are subharmonic in $U$.
\end{definition}

Many properties of harmonic and subharmonic functions known from the classical potential theory hold also in the Berkovich setting. We now list those which will be used later in the paper.\\

\begin{proposition} (Maximum principle; Proposition 8.14 (B) in \cite{BR10}) \label{prop: maxprinciple} Let $U \subsetneq \sP^1$ be a domain and let $f$ be a subharmonic function in $U$. Then, if $M \in \mathbb{R}$ is such that, for each $y \in \partial U$, $\limsup_{U \ni z \to y} f(z) \leq M$, then $f(z) \leq M$ for all $z \in U$.
\end{proposition}

\begin{proposition} (Proposition 8.26 (E) in \cite{BR10})  \label{prop: netsbounded} Let $U \subset \sP^1$ be an open set and let $\langle f_\alpha\rangle_{\alpha \in A}$ be a net of subharmonic functions in $U$ which is locally uniformly bounded from above. Put $f(x)=\sup_{\alpha \in A} f_\alpha(x)$. Then $f^*$ is subharmonic in $U$ and $f^*(x)=f(x)$ for all $x \in U \cap \sH^1$.
\end{proposition}

\begin{proposition} (Harnack's inequality; Lemma 8.33, \cite{BR10}) \label{prop: Harnackineq} Let $U \subset \sP^1$ be a domain. then for each $x_0 \in U$ and each compact set $X \subset U$ there is a constant $C=C(x_0,X)$ such that for each function which is harmonic and nonnegative in $U$, each $x \in X$ satisfies
\[
(1/C)\cdot h(x_0) \leq h(x) \leq C \cdot h(x_0).
\]

\end{proposition}

The default topology on the set $SH(U)$ of subharmonic functions in a domain $U \subset \sP^1$ is that of pointwise convergence on $U \cap \sH^1$. The space $\mathcal{M}^+(U)$ of positive locally finite Borel measures on $U$ can be given the following topology: a net of measures  $\langle \mu_\alpha\rangle_{\alpha \in A}$ in $\mathcal{M}^+(U)$ converges (weakly) to a measure $\mu \in \mathcal{M}^+(U)$ if and only if $\int f d\mu_\alpha \to \int f d\mu$ for all $f \in \mathcal{C}_c(U)$. Here $\mathcal{C}_c(U)$ is the space of all continuous functions $f: U \to \mathbb{R}$ for which there exists a compact subset $X_f \subset U$ such that $f|_{U\setminus X_f}\equiv 0$. The following continuity result will be used.

\begin{proposition} (Theorem 8.44, \cite{BR10})  \label{prop: contLaplacian} Let $SH(U)$ and $\mathcal{M}^+(U)$ be the topological spaces as described above. Then  the operator $-\Delta_U: SH(U)\to \mathcal{M}^+(U)$ is continuous.
\end{proposition}

Here we also recall, in a general framework, a useful consequence of  weak convergence.

\begin{proposition} \label{prop: weakconv} (cf. Formula (5.1.15) in \cite{AGS} for sequences of positive Radon measures in metric spaces) Let $X$ be a locally compact Hausdorff space, $f$ be a nonnegative lower semicontinuous function on $X$, and $\langle \mu_\alpha\rangle_{\alpha \in A}$ be a net of positive Radon measures  on $X$ converging weakly to a positive Radon measure $\mu$  on $X$. Then 
\[
\liminf_\alpha \int f d\mu_\alpha \geq \int fd\mu.
\] 
\end{proposition}
\begin{proof} By Proposition A.3 in \cite{BR10}, 
\[
\int f d\mu =\sup \bigl \{\int g d\mu: \ g \in \mathcal{C}_c(X), \ 0 \leq g \leq f \bigr \}, 
\]
where the class $\mathcal{C}_c$ is defined as above. Then 
\begin{multline*}
\liminf_\alpha \int f d\mu_\alpha \geq \sup_{g \in \mathcal{C}_c(X), 0 \leq g \leq f} \liminf_\alpha \int g d\mu_\alpha = \\ \sup_{g \in \mathcal{C}_c(X), 0 \leq g \leq f}  \int g d\mu = \int f d\mu .
\end{multline*}
\end{proof}

In the literature on subharmonic functions in the Euclidean space various  (related) results are referred to as a  ``Hartogs lemma''. Some of them have their non-archimedean counterparts.

\begin{proposition} (Proposition 8.54  in \cite{BR10}) \label{prop: Hartogs} Let $U \subset \sP^1$ be a domain and let  $\{g_n\}$ be a sequence of functions subharmonic in $U$. Suppose that the functions $g_n$ are uniformly bounded from above in $U$. Then  one of the following holds: \\
(i) there is a subsequence $\{g_{n_k}\}$ which converges uniformly to $-\infty$ on each compact subset of $U$, or 
(ii) there is a subsequence $\{g_{n_k}\}$ and a function $G$ subharmonic in $U$ such that $g_{n_k}$ converge pointwise to $G$ in $U \cap \sH^1$ and such that for each continuous function $f:U \to \mathbb{R}$ and each compact subset $X \subset U$,
\[
\limsup_{k \to \infty} \bigl( \sup_{z \in X} (g_{n_k}(z)-f(z))\bigr) \leq \sup_{z\in X}(G(z)-f(z)).
\]
\end{proposition}

\begin{coro} (cf. Theorem 1.31  and Corollary 1.32 in \cite{LG86} in the Euclidean case; see also \cite{Klimek91}, Theorem 2.6.4)) \label{coro: littleHartogs}  Let $U \subset \sP^1$ be a domain and let  $\{g_n\}$ be a sequence of functions subharmonic in $U$ uniformly bounded from above in $U$. Suppose  (possibly after passing to a subsequence) that $g_n$ converge pointwise on $U \cap \sH^1$ to a function $G$ subharmonic in $U$. \\
(i) Suppose further that for some compact set $X \subset U$ there is a constant $C=C(X)$ such that $(\limsup_{n \to \infty}g_n(x))^* \leq C$ on $X$. Then for every $\varepsilon >0$ there is an $n_0=n_0(\varepsilon,X)$ such that for every $n \geq n_0$ and every $x \in  X$ we have $g_n(x) \leq C+\varepsilon$.
(ii) More generally, under the above asumptions, let $F$ be a function continuous on $X$ such that  $(\limsup_{n \to \infty}g_n(x))^* \leq F(x)$ on $X$. Then for every $\varepsilon >0$ there is an $n_0=n_0(\varepsilon,X)$ such that for every $n \geq n_0$ and every $x \in  X$ we have $g_n(x) \leq F(x) +\varepsilon$.
\end{coro}
\begin{proof} For (i), let $c_n:=\sup_{z \in X}g_n(z)$ and let $\varepsilon >0$.  By  the inequality from part (ii) of Proposition \ref{prop: Hartogs},  $(\limsup_{n \to \infty}g_n(x))^* \leq C$ on $X$ with $C:=sup_{z \in X}G(z)$. Take an $n_0$ such that for every $n \geq n_0$, $c_n \leq \limsup_{n \to \infty}c_n +\varepsilon$. Then $c_n \leq C+\varepsilon$. Part (ii) follows from (i) by observing that $F$ is uniformly continuous on $X$ (taking into account the unique uniform structure on $\sP^1$ compatible with the Berkovich topology; see Appendix A.9, \cite{BR10}).
\end{proof}

Here are some results on convergence of harmonic functions in open subsets of $\sP^1$.

\begin{proposition} (Proposition 7.31, \cite{BR10}) \label{prop: harmconv} Let $U$ be an open subset of $\sP^1$. Suppose $f_1,f_2,...$ are harmonic on $U$ and converge pointwise to a function $f: U \to \mathbb{R}$. Then $f$ is harmonic in $U$ and the $f_i$ converge uniformly to $f$ on compact subsets of $U$.
\end{proposition}

Proposition 7.34 in \cite{BR10}  gives a non-archimedean Harnack principle for a sequence $0 \leq h_1 \leq h_2 \leq ...$ of harmonic functions on a domain $\Omega \subset \sP^1$.  Theorem 4.7.2 in \cite{Wanner2016} is slightly more general, since it does not require the sequence to be non-negative.  A similar result is  Proposition 3.1.2 in \cite{ThuillierThesis}, while Proposition 3.1.3 and Corollary 3.3.10 in \cite{ThuillierThesis} are equicontinuity results for locally uniformly bounded families of harmonic functions. Here we offer an even more general variant of the Harnack principle, for a family $\mathcal{F}$ of harmonic functions on a domain $\Omega \subset \sP^1$ that is  locally uniformly bounded from below in $\Omega$ but not necessarily uniformly bounded from below in $\Omega$, countable or increasing. This is a  new result in the non-archimedean setting.

\begin{proposition} (cf. \cite{AG}, Theorem 1.5.11 in the classical case) \label{prop: harmlocunifbound} Let $\Omega \subset \sP^1$ be a domain and let $\mathcal{F}$ be a family  of harmonic functions on $\Omega$ which is locally uniformly bounded from below. Then either $\sup \mathcal{F} \equiv +\infty$ in $\Omega$ or  $\mathcal{F}$ is uniformly bounded and uniformly equicontinuous on every compact subset $E \subset \Omega$.
\end{proposition}

\begin{proof} If $\sup \mathcal{F} \not \equiv +\infty$, then we can fix an $x_0 \in \Omega$ such that  $\sup \mathcal{F}(x_0)< +\infty$. Let $E \subset \Omega$ be a compact set and let $V \subset \sP^1$ be a domain such that $E \cup \{x_0\} \subset V \subset \overline{V} \subset \Omega$. The family $\mathcal{F}$ is uniformly bounded from below on $\overline{V}$, so without loss of generality we can assume that all functions in $\mathcal{F}$ are positive on $V$. By Proposition \ref{prop: Harnackineq}, there is a constant $C=C(E,x_0)$ such that, for every $f \in \mathcal{F}$ and every $x \in E$, $0 < f(x) < C < +\infty$. Hence $\mathcal{F}$ is uniformly bounded on $E$. Further, the right side of the inequality in Proposition \ref{prop: Harnackineq} is the same as Axiom $\mathrm{III}_2$  of \cite{LW65}.  By the Theorem of \cite{LW65}, the family $\mathcal{F}$ of positive functions is equicontinuous in a compact neighborhood of $x_0$.  
\end{proof}

\subsection{The Green function}\label{sec: Green}

In this section, we fix the base $q >1$ of logarithms which occur in the definition of potentials. Typically this is done so that the absolute value $|\cdot|$ coincides with the modulus of the  Haar measure on the additive group of $K$. When $K=\mathbb{C}_p$,   $p$ prime, one takes $q=p$ (see Example 6.4 in \cite{BR10}).\\

\begin{definition} (\cite{BR10}, Section 6.3)   Fix $\zeta \in \sP^1$ and let $\nu$ be a positive measure on $\sP^1$. The potential of $\nu$ with respect to $\zeta$ is the function 
\[
u_\nu(x,\zeta)=\int -\log [x,y]_\zeta d\nu(y).
\]
\end{definition}

Remark (Example 8.8, \cite{BR10}): If $\nu$ is a probability measure on $\sP^1$ and $\zeta \not \in {\rm supp }\  \nu$, then the potential $u_\nu(\cdot,\zeta)$ is strongly subharmonic in $\sP^1 \setminus {\rm supp }\  \nu$, while $-u_\nu(\cdot,\zeta)$ is strongly subharmonic in $\sP^1 \setminus \{\zeta\}$.  

\begin{proposition}
(Theorem 8.38, \cite{BR10}: Riesz decomposition theorem) \label{prop: Riesz} Let $V$ be a simple subdomain of an open set $U \subset \sP^1$. Fix $\zeta \in   \sP^1 \setminus \overline{V}$. Suppose $f$ is subharmonic in $U$
 and let $\nu$ be the positive measure $\nu=-\Delta_V(f)$. Then there is a function $h_V$ which is continuous in $\overline{V}$, harmonic in $V$ and such that $f(z)=h_V(z)-u_\nu(z,\zeta)$ for all $z \in \overline{V}$.

\end{proposition}

Let now $E \subset \sP^1$ be a compact subset such that $\zeta \not \in E$ and let $\nu$ is a probability measure with support contained in $E$. Following Section 6.1, \cite{BR10}, we define a few basic notions.

\begin{definition}  The energy of $\nu$ with respect to $\zeta$ is 
\[
I_\zeta(\nu)=\int_E u_\nu(x,\zeta) d\nu(x).
\]
\end{definition}

Varying $\nu$ over all probability measure with support in $E$ we further define

\begin{definition} The Robin constant of $E$ with respect to $\zeta$ is 
\[
V_\zeta(E)=\inf_\nu I_\zeta(\nu)
\]
and the (logarithmic) capacity of $E$ is 

\[
{\rm cap }_\zeta(E)=q^{-V_\zeta(E)}.
\]
\end{definition}

Remark: Unlike in the classical setting, there are finite subsets of Berkovich line with positive capacity: in fact, for every $a \in \sH^1$ and $\zeta \neq a$, ${\rm cap }_\zeta(\{a\})={\rm diam}_\zeta(a) >0$.\\

Remark: Other capacitary notions can be defined for subsets of $\sP^1$ (see Section 6.4 of \cite{BR10}). While they coincide for compact subsets (Theorem 6.24 in \cite{BR10}), they may differ for non-compact subsets (Remark 6.25 in \cite{BR10}).\\

\begin{proposition}\label{prop: countsum}
(Corollary 6.17, \cite{BR10}): Let $\{e_n\}_{n \in \mathbb{N}}$ be a countable collection of Borel subsets of $\sP^1\setminus\{\zeta\}$ such that each $e_n$ has capacity $0$. Let $e=\bigcup_{n\in \mathbb{N}}e_n$. Then $e$ has capacity $0$.
\end{proposition}

If $E$ is compact and ${\rm cap }_\zeta(E)>0$, there is a (unique) probability measure $\mu_\zeta$ supported on $E$ for which $V_\zeta(E)=I_{\zeta}(\mu_\zeta)$ (Propositions  6.6 and 7.21 in \cite{BR10}). 

\begin{definition} 
The measure $\mu_\zeta$ is called the equilibrium measure of $E$ with respect to $\zeta$. If $\zeta \not \in E$,   the Green function of $E$ with respect to $\zeta$ is 
\[
G(z,\zeta,E):=V_\zeta(E)-u_{\mu_\zeta}(z,\zeta).
\]
\end{definition}

We recall here several known properties of  the Green function of a compact set $E \subset \sP^1\setminus\{\zeta\}$ with positive capacity. Let $D_\zeta$ denote the connected component of $\sP^1\setminus E$ containing $\zeta$.

\begin{proposition} (Proposition 7.37, \cite{BR10}) \label{prop: Green} (i) $G(z,\zeta,E)$ is finite for every $z \in  \sP^1 \setminus\{\zeta\}$.\\
(ii) $G(z,\zeta,E) \geq 0$ for every $z \in  \sP^1$ and $G(z,\zeta,E) >  0$ 
for every $z \in D_\zeta$.\\
(iii) $G(z,\zeta,E) = 0$ on $\sP^1 \setminus D_\zeta$ except on a (possibly empty) set $e \subset \partial D_\zeta$ of capacity zero.\\
(iv) $G(z,\zeta,E)$ is continuous on $\sP^1 \setminus e$.\\
(v) $G(z,\zeta,E)$ is subharmonic on $\sP^1\setminus\{\zeta\}$ and (strongly) harmonic on $D_\zeta \setminus\{\zeta\}$. For every $a \neq \zeta$, $G(z,\zeta,E)-\log [z,a]_\zeta$ extends to a function harmonic in a neighborhood of $\zeta$.\\
(vi) $G(z,\zeta,E)=G(z,\zeta, \sP^1\setminus D_\zeta)= G(z,\zeta, \partial D_\zeta)$.\\
(vii) If $E_1 \subset E_2$ are two sets of positive capacity, then $G(z,\zeta,E_1) \geq G(z,\zeta,E_2)$.
\end{proposition}

\section{Brelot-Cartan principle}\label{sec: Cartan}

In this section we will  work under an additional topological assumption,  namely that of the second countability of the Berkovich topology. This is restrictive but still reasonable. On one hand, when there exists a countable base of open sets for the Berkovich topology on $\sP^1$,    the underlying valued field is  necessarily separable as a topological space. The precise conditions characterizing such fields are not known but, as observed in \cite{MO12}, separability may fail even if the residue field and the value group are both countable. On the other hand,  for the fields $\mathbb{C}_p$, $p$ prime (which are separable; see discussion before Corollary 1.20 in \cite{BR10}), such a countable base exist.    For some results related to  countability of the  topologies on $\sP^1$ see \cite{Favre11} and \cite{HLP14}.
The second countability of the Berkovich topology allows us to apply  the following lemma: 
\\
 
\begin{lemma} (Choquet topological lemma; for this formulation and its proof see  \cite{Doob01}, A.VIII.3) \label{lemma: Choquet} Let $\{u_\iota\}_{\iota \in I}$ be a family of functions from a second countable Hausdorff space to $[-\infty,+\infty]$. For a $J \subset I$, define $u^J=\sup_{\iota \in J}u$. Then there is a countable subset $J \subset I$ such that $(u^J)^*=(u^I)^*$, where $v^*(x)=\limsup_{y \to x}v(y)$.
\end{lemma}

Using  the Choquet topological lemma we can prove the Brelot-Cartan principle. It was first proved in \cite{Brelot38} (in connection with Dirichlet regularity) and \cite{Cartan45} (in a more general context).  Our formulation and the general strategy of proof  follows the classical version as presented in \cite{AG}, Theorem 5.7.1 (ii) and (iii).

\begin{theorem} \label{prop: Cartanprinciple}  Let $\Omega \subsetneq \sP^1$ be a domain and  $\{u_\iota\}_{\iota \in I}$ be a family of functions subharmonic on $\Omega$ which is locally uniformly bounded from above. Let $u=\sup_{\iota \in I}u_\iota$. Then the set $\{x \in \Omega: u(x) < u^*(x)\}$ has capacity zero.
\end{theorem}
 
\begin{proof}  By Proposition \ref{prop: exhaustion}, every domain different from $\sP^1$ can be exhausted by a sequence $V_1 \subset V_2 \subset...$ of (strict) simple domains such that $\overline{V_n} \subset V_{n+1} \subset \Omega$ for each $n$. Since $u(x) = u^*(x)$  in $\Omega \cap \sH^1$, it is enough to prove that for any simple domain $V \subset \Omega$ any compact subset of   $\{x \in V: u(x) < u^*(x)\}$  has capacity zero. So let us fix a point $\zeta \not \in \Omega$ and take a simple domain $V$ such than $\overline{V} \subset \Omega$. By Lemma \ref{lemma: Choquet} there is a sequence $(u_n) \subset \{u_\iota\}_{\iota \in I}$ such that $u^*=v^*$, where $v=\sup_{n \in \mathbb{N}}u_n$. By Proposition \ref{prop: netsbounded}, $v^*$ is subharmonic.  Defining $v_n=\max\{u_1,...,u_n\}$ we get an increasing sequence $(v_n)$ of subharmonic functions with limit $v$ in $\Omega$. By Proposition \ref{prop: contLaplacian}, the positive measures $\mu_n:=-\Delta_V(v_n)$  converge  weakly  to the positive bounded measure $\mu:=-\Delta_V(v^*)$. \\
Consider the potentials $p_{\mu_n}(z,\zeta)=\int \! -\log \delta(z,w)_\zeta d\mu_n(w)$. Since the function $ -\log \delta(z,w)_\zeta$ is lower semicontinuous and bounded below on $\overline{V}$, we have by Proposition \ref{prop: weakconv} that $\liminf_{n \to \infty}p_{\mu_n} \geq p_\mu$ on $V$. \\
ByProposition \ref{prop: Riesz}, for every $n$ there exists a  function $h_n$ continuous on  $\overline{V}$ and harmonic on $V$ such that $v_n=h_n - p_{\mu_n}$. The  functions $h_n, n=1,2,...$ are locally uniformly bounded  below on $V$. Indeed, $h_n \geq v_1+ p_{\mu_n}$ for every $n$.  If $C$ is a compact subset of $V$  and $x \in C$, then there is an $N >0$ such that for every $n \geq N$, $h_n(x) \geq v_1(x) + \liminf_{n \to \infty} p_{\mu_n} -\frac{1}{2} \geq  (v_1)_*(x) + p_\mu(x)-\frac{1}{2}$. Taking infimum over $x \in C$ on both sides of the inequality, we get $h_n \geq \inf_C (v_1)_*+ \inf_C p_\mu \cdot \mu(C) -\frac{1}{2}$ on $C$ for every $n \geq N$, and hence a common lower bound on $C$ for all $h_n$. By  Proposition \ref{prop: harmlocunifbound}, there exists a subsequence $(h_{n_k})$ that converges locally uniformly to a (continuous) harmonic function $h$ on $V$. Since $v_{n_k} \to v$ and $h_{n_k} \to h$,  then also $p_{\mu_{n_k}}$ converge  as $k \to \infty$ and $\lim_{k \to \infty} p_{\mu_{n_k}} \geq p_\mu$ on $V$. \\
 We have $v(x) = h(x) - \lim_{k \to \infty} p_{\mu_{n_k}}(x) \leq h(x) - p_\mu(x) = v^*(x)$ in $V$,  so it is enough to prove that the set $E=\{x \in V: \lim_{k \to \infty} p_{\mu_{n_k}}(x) > p_\mu(x) \}$ has capacity zero. Suppose there is a compact subset $F \subset E$ such that ${\rm cap}_\zeta F >0$. Let $\nu$ be the equilibrium measure on $F$ with respect to $\zeta$. Then, by continuity of $p_\nu$ on $F' \subset F$ with ${\rm cap}_\zeta F >0$, portmanteau theorem and Fatou's lemma, 
 \begin{multline*}
 \int_{F'} p_\mu d \nu \!=\! \int_{F'} p_\nu d \mu \!=\!  \lim_{k \to \infty} \int_{F'} p_\nu d \mu_{n_k} \!=\!   \lim_{k \to \infty} \int_{F'} p_{\mu_{n_k}} d\nu \! \\
 \geq \! \int_{F'}  \lim_{k \to \infty} p_{\mu_{n_k}} d\nu\!>\! \int_{F'} p_\mu d\nu,
 \end{multline*}
 which is a contradiction. Hence $E$ has capacity zero. Since $v \leq u$ and $u^*=v^*$, the set $\{x: u(x) < u^*(x)\}$ is a subset of $\{x: v(x) < v^*(x)\}$ and also has capacity zero.
\end{proof}

\begin{remark}
In Proposition 8.26 (E) of \cite{BR10} it was shown that $u(x)=u^*(x)$ for every $x \in \sH^1$. The Brelot-Cartan principle in the Berkovich setting refines this relation, since nonempty subsets of $\sH^1$ have positive capacity. \\
\end{remark}

\section{The extremal function}\label{sec: extremal}

\subsection{An upper envelope}\label{sec: envelope}

 Fix  a point $\zeta \in \sP^1$. Define the following: 

\begin{definition} \label{defin: classL}
\begin{multline*} \mathcal{L}= \mathcal{L}(\zeta):=\bigl\{u: u\text{ is subharmonic in }\sP^1\setminus\{\zeta\},  \\ \forall a \ne \zeta \   u -\log [z,a]_\zeta \text{ extends to a function subharmonic in } \sP^1\setminus\{a\} \bigr\}. 
\end{multline*}

\end{definition}

\begin{definition} \label{defin: Qfcn}  Let  $\zeta$ and $\mathcal{L}(\zeta)$ be as in Definition \ref{defin: classL} and let $E \subsetneq \sP^1 \setminus \{\zeta\}$ be nonempty. Define 
\[
\cS\mapsto Q_E(\cS)
:=\sup \{u(\cS):u \in \mathcal{L}, u|E\leq 0\}.
\]
We call $Q_E$ the $\mathcal{L}$-extremal function (relative to $\zeta$) associated to $E$. 
\end{definition}

\begin{proposition} \label{prop: nonempty} If $E$ has positive capacity, then the class of functions defining $Q_E$ is not empty.
\end{proposition}

\begin{proof} If $E \subset \sH^1$, then the Green function $G(\cdot,\zeta,E)$ of $E$ relative to $\zeta$ is in $\mathcal{L}$ and $G(\cdot,\zeta,E)=0$ on $E$. If $E \cap \mathbb{P}^1 \neq \emptyset$, take a point $a \in E \cap \mathbb{P}^1$ and $r \geq \sup_{z\in E}[z,a]$. The function  $(\log[z,a]_{\zeta}-\log r)^+:=\max \{\log[z,a]_{\zeta}-\log r,0\}$ is then in $\mathcal{L}$, and it is non-positive on $E$.  
\end{proof}

\noindent The following property holds (for the classical analog, see \cite{Klimek91}, Corollary 5.1.2).\\

\begin{proposition} \label{prop: nested} If $E_1 \supset E_2 \supset ...$ is a sequence of nonempty sets in $\sP^1 \setminus \{\zeta\}$ and $E=\bigcap_{n \to \infty}E_n$, then $Q_E =\lim_{n \to \infty}Q_{E_n}$.
\end{proposition}

\begin{proof} Since $Q_{E_1} \leq ... \leq Q_{E_n} \leq ...\leq Q_E$, the limit $\lim_{n \to \infty}Q_{E_n}$ exists and  does not exceed $Q_E$. Conversely, for an $u \in \mathcal{L}$  and an $\varepsilon > 0$ the open set $\{ u < \varepsilon\}$ is a neighborhood of $E$ containing $E_n$ for all sufficiently large $n$. Hence for those $n$, $u-\varepsilon \leq Q_{E_n} \leq \lim_{n \to \infty}Q_{E_n}$. Since $\varepsilon$ was arbitrary, $Q_E \leq \lim_{n \to \infty}Q_{E_n}$.
\end{proof}

Having fixed a point $\zeta \in \sP^1$, we can compute an example of the function $Q$. Let $a  \in \sP^1\setminus\{\zeta\}$, $r \in ({\rm diam}_\zeta(a), {\rm diam}_\zeta(\zeta))$ and let $B_r=B(a,r)_\zeta:=\{z: [z,a]_\zeta\leq r\}$. When $a \in \sH^1$, the bounds on $r$ guarantee that  $B_r \ne \emptyset$ and $B_r \ne \sP^1$.\\

 \begin{proposition} \label{prop: pseudoball} (cf. Property 2.6 in \cite{Siciak81}, Example 5.1.1 in \cite{Klimek91}, Th\'eor\`eme 3.6 in \cite{Ze91}) 
 \[
Q_{B_r}(z)=\max \{\log[z,a]_{\zeta}-\log r,0\}=(\log[z,a]_{\zeta}-\log r)^+.\]
\end{proposition}

\begin{proof} Recall that, by Formula 5 in Section 5.2 of \cite{Rumely89}, the function on the right-hand side is the Green function of $B_r$.   Arguing like in Proposition \ref{prop: nonempty}, we get that  $(\log[z,a]-\log r)^+ \leq Q_{B_r}(z)$. In the reverse direction, let $v \in \mathcal{L}=\mathcal{L}(B_r,\zeta)$.  The function $v(z) - \log [z,a]_\zeta + \log r$ extends to a subharmonic function  in $A_r=\sP^1 \setminus B_r$.  By Remark 4.13 in \cite{BR10}), the set $A_r$  is the connected component of $\sP^1 \setminus \{x_r\}$ containing $\zeta$, where $x_r$ is the unique point on the path from $a$ to $\zeta$ with ${\rm diam}_\zeta(x_r)=r$. Moreover,  $v(z) - \log [z,a]_\zeta + \log r \leq 0$ on $\partial B_r=\partial A_r$,  and so by the maximum principle $v(z) - \log [z,a]_\zeta + \log r \leq 0$ in $A_r$. Hence  on $\sP^1 \setminus \{\zeta\}$ we have $v(z) \leq \log [z,a]_\zeta - \log r$ and  finally 
$Q_{B_r} \leq (\log[z,a]-\log r)^+$.
\end{proof}

The behavior of locally uniformly bounded families in $\mathcal{L}$ is analogous to what happens in $\mathbb{C}^N$. Namely, the following holds:

\begin{proposition} (cf. \cite{Klimek91}, Proposition 5.2.1; \cite{Siciak81}, Theorem 3.5; \cite{Ze91}, Lemme 3.10) \label{prop: uniffamilies} Let $\mathcal{U} \subset \mathcal{L}=\mathcal{L}(\zeta)$ be a non-empty family,   let $u=\sup\{v: v \in \mathcal{U}\}$ and let $\Omega$ be a  connected component  of $\sP^1\setminus\{\zeta\}$. If the set $A=\{z \in \Omega: u(z)<+\infty\}$ has  nonzero capacity, then the family $\mathcal{U}$ is locally uniformly bounded from above in $\Omega$. If moreover $\mathcal{U}$ is locally uniformly bounded from above in $\sP^1\setminus\{\zeta\}$, then $u^*\in \mathcal{L}$.
\end{proposition}

\begin{proof}  Suppose $\mathcal{U}$ is not locally uniformly bounded in $\Omega$. Then there exists  a simple domain $V \subset \overline{V} \subset \Omega$  and a sequence $(u_j)$ in  $\mathcal{U}$ such that $m_j:=\sup_{\overline{V}}u_j \geq 2^j$ for every $j \in \mathbb{N}$. We have $u_j \leq m_j + Q_{\overline{V}}$ on $\sP^1\setminus \{\zeta\}$ for every $j \geq 1$. Fix  a  simple domain  $V'$ such that $\overline{V} \subset V' \subset \Omega$. The family $(u_j -m_j)_{j \in \mathbb{N}}$ is uniformly bounded in $V'$.  We claim that there exists an $x_0 \in V'$  such that $\limsup_{j \to \infty} u_j(x_0) >-\infty$. Suppose to the contrary that  $\limsup_{j \to \infty} u_j \equiv -\infty$ in $V'$. Then $\limsup_{j \to \infty} \exp [u_j(x_0)-m_j] \equiv 0$ in $V'$, so $\lim_{j \to \infty} \exp [u_j(x_0)-m_j] \equiv 0$ in $V'$. Note that  all  functions $\exp [u_j(x_0)-m_j]$ are subharmonic (by Corollary 8.29 in \cite{BR10}) and have an uniform upper bound in $V'$.    Hence, by Proposition \ref{prop: Hartogs} and Corollary \ref{coro: littleHartogs}, there exists a subsequence $(u_{j_n})$ and an $n_0 \geq 1$ such that $\sup_{V'} \exp [u_{j_n}(x_0)-m_{j_n}] \leq 1/2$  for every $n \geq n_0$. Taking natural logarithm of both sides of this inequality we get a contradiction with the definition of $m_j$. The claim allows us (passing to a subsequence of $u_j$ if necessary) to fix an $x_0 \in \Omega$ and an $\varepsilon >0$  such that  $u_j(x_0)-m_j > \log \varepsilon$ for every $j \geq 1$.  Define next 
\[
v(x)=\sum_{j=1}^{\infty}\frac{1}{2^j}(u_j(x)-m_j)
\]
for every $x \in \sP^1\setminus \{\zeta\}$. 
Note that on every simple subdomain of $\Omega$ the function $v$ is the limit of a uniformly convergent sequence of subharmonic functions. Therefore (by Proposition 8.26 (C) in \cite{BR10}) $v$ is subharmonic in $\Omega$. 
If $x \in A \cap \Omega$, then $\sup_j u_j(x) <+\infty$, and so $v(x)=-\infty$, while $v(x_0) \geq \log \varepsilon >-\infty$. Hence, by Corollary 8.40 in \cite{BR10}, $A \cap \Omega$ has capacity zero.\\
 Now, if  the family $\mathcal{U}$ is locally uniformly bounded in $\sP^1 \setminus \{\zeta\}$, then by Proposition \ref{prop: netsbounded}, $u^*$ is subharmonic on   $\sP^1 \setminus \{\zeta\}$.  Let us fix arbitrarily an $a \in \sP^1 \setminus \{\zeta\}$ and a simple subdomain $W$ of $\sP^1 \setminus \{a\}$. Let $h$ be a harmonic function in $W$ such that $u^* -\log[\cdot,a]_\zeta \leq h$ on $\partial W$.  Then for every $v \in \mathcal{U}$, $v -\log[\cdot,a]_\zeta \leq h$ on $\partial W$.  By Theorem 8.19 in \cite{BR10} or Corollary 3.1.12 in \cite{ThuillierThesis} , $v -\log [\cdot,a]_\zeta \leq  h$ in $W$ (when $\zeta \in W$, we consider the subharmonic extension of $v -\log[\cdot,a]_\zeta$ to $W$), so $u   \leq \log [\cdot,a]_\zeta + h$ in $W$.  Since  $\log[\cdot,a]_\zeta + h$ is continuous as a  function of $z \in W$, we also have $u^*   \leq \log[\cdot,a]_\zeta + h$ in $W$ and so (again by Theorem 8.19 in \cite{BR10}), $u^*-\log[\cdot,a]_\zeta$ extends as a subharmonic function in $W$.  Hence $u^* \in \mathcal{L}$.
  
  \end{proof}
  
  \begin{coro}\label{coro: polarsets} (cf. Corollary 3.9 in \cite{Siciak81}, Theorem 5.2.4 in \cite{Klimek91})
  For  a nonempty $E \subset \sP^1 \setminus \{\zeta\}$ the following are equivalent:\\
  (i) There exists a function $v$ subharmonic in $\sP^1 \setminus \{\zeta\}$ such that $E \subset \{x: v(x)=-\infty\}$.\\
  (ii) The capacity of $E$ equals zero.\\
  (iii) There exists a function $w \in \mathcal{L}(\zeta)$ such that $E \subset \{x: w(x)=-\infty\}$.
  \end{coro}
  \begin{proof} The implication $(iii) \Rightarrow (i)$ is obvious and $(i) \Rightarrow (ii)$ is Corollary 8.40 in \cite{BR10}. Applying Proposition \ref{prop: uniffamilies} to the family $\mathcal{U}=\{w \in \mathcal{L}: w|_E \leq 0\}$ yields $(ii) \Rightarrow (iii)$.
  \end{proof}
  
  Remark: For a compact set $F \subset  \sP^1\setminus\{\zeta\}$ one can also directly construct a function $w$  of class  $\mathcal{L}(\zeta)$ such that $F \subset \{w=-\infty\}$.  For $N \in \mathbb{N}$ let $P_N$ be a pseudopolynomial $P_N(z)=\prod_{i=1}^N [z,a_i]_\zeta$ with $a_1, ..., a_N \in F$. For every $k \in \mathbb{N}$, let us pick $a_1,...,a_{N_k} \in F$ and a real number $p_k \geq 0$ such that $\sum_{k=1}^\infty p_k=1$. Let further 
\[
w(z):=\sum_{k=1}^\infty p_k\bigl (\frac{1}{N_k}\log P_{N_k}(z)\bigr ). 
\]
 The function $w$ is the negative of an Evans potential for $F$ (see Lemma 7.18 in \cite{BR10}). It is easy to see that $w \in \mathcal{L}(\zeta)$.\\

Let $E \subset \sP^1 \setminus \{\zeta\}$ be an arbitrary compact  with positive capacity and let $G(\cdot,\zeta,E):= V_E-p_{\nu_\zeta}(\cdot)$ be the Green function with pole at $\zeta$ (see subsection \ref{sec: Green}).   Our goal is  to show that the equality $Q_E^*=G(\cdot,E,\zeta)$ holds. From Proposition \ref{prop: pseudoball} we already know it for a class of sets $B_r$. Now we will  establish it in another  important special case, that of the complement of a simple domain. Note that in these special cases $Q_E$ is continuous, so $Q_E=Q_E^*$.

\begin{proposition} \label{prop: simpledom} Let $V \subset \sP^1$ be a simple domain such that $\zeta \in V$. Then $Q_{\sP^1 \setminus V}(\cdot) = G(\cdot, \sP^1 \setminus V, \zeta)$. In particular, the function $Q_{\sP^1 \setminus V}$ is continuous in $\sP^1\setminus\{\zeta\}$ and strongly harmonic in $V\setminus \{\zeta\}$.
\end{proposition}

\begin{proof}  
For $z \in  \sP^1 \setminus V$, $Q_{\sP^1 \setminus V}(z) \leq 0 = G(z,\sP^1 \setminus V, \zeta)$. Fix an $a \not \in \overline{V}$. By Proposition \ref{prop: uniffamilies}, the function  $Q_{\sP^1 \setminus V}^* - \log [\cdot,a]_\zeta$ has subharmonic extension to $V$. Note that $Q_{\sP^1 \setminus V}=Q_{\sP^1 \setminus V}^*$ on $\partial V$, since $\partial V \subset  \sH^1$. Then   $Q_{\sP^1 \setminus V}^*(z) - \log [z,a]_\zeta - (G(z, \sP^1 \setminus V, \zeta)- \log [z,a]_\zeta) \leq 0$ on $\partial V$. Using the maximum principle we get that $Q_{\sP^1 \setminus V}(\cdot) \leq  G(\cdot, \sP^1 \setminus V, \zeta)$  in $V \setminus \{\zeta\}$. In the other direction, $G(\cdot, \sP^1 \setminus V, \zeta) \in \mathcal{L}(\sP^1 \setminus V, \zeta)$ and $G(\cdot, \sP^1 \setminus V, \zeta)|_{\sP^1 \setminus V}=0$. Hence in $\sP^1\setminus\{\zeta\}$ we have  $Q_{\sP^1 \setminus V}(\cdot) = G(\cdot, \sP^1 \setminus V, \zeta)$. 
\end{proof}

Now we are ready to prove the general case.

\begin{theorem} \label{theorem: equality} For an arbitrary  compact set $E \subset \sP^1\setminus \{\zeta\}$ of positive capacity, $G(\cdot,\zeta,E)= Q_E^*(\cdot)$ in $\sP^1\setminus \{\zeta\}$.
\end{theorem}

\begin{proof}  Note first that $Q_{\sP^1\setminus D_\zeta} = Q_E$, where $D_\zeta$ is the connected component of $\sP^1\setminus E$ containing $\zeta$. Indeed, let $U$ be a connected component of $\sP^1\setminus E$ not containing $\zeta$ and let $u \in \mathcal{L}(E,\zeta)$. By the maximum principle, $u \leq 0$ on $U$. Hence $u \leq 0$ on $\sP^1\setminus D_\zeta$, which shows that $Q_{\sP^1\setminus D_\zeta} \geq Q_E$ (the other inequality is obvious). Since $G(\cdot,\zeta,E)=G(\cdot,\zeta,\sP^1\setminus D_\zeta)$, it is thus enough to establish the equality in the theorem for sets $E=\sP^1\setminus D$ of positive capacity, where $D \subsetneq \sP^1$ is a domain containing $\zeta$.\\

For such a domain $D$  consider an exhaustion  by simple domains $V_1 \subset V_2 \subset ...$  such that $\overline{V_n} \subset V_{n+1}$ and let $F_n := \sP^1\setminus V_n$. Then, by Proposition \ref{prop: simpledom}, $G_n:= G(\cdot,\zeta,F_n)= Q_{F_n}$. By Proposition  \ref{prop: nested}, $Q_E= \lim_{n \to \infty}G_n \leq G(\cdot,\zeta, E)$.   \\

It remains to prove  that $G(\cdot,\zeta,E) \leq Q_E^*$ in $\sP^1\setminus \{\zeta\}$. Let $e \subset \partial D$ be the (possibly empty) set of capacity zero such that $G(z,\zeta,E) >0$ for all $z \in e$. Then $G(\cdot,\zeta,E) \leq Q_{E\setminus e}$. By Corollary \ref{coro: polarsets}, there exists a function $v \in \mathcal{L}$ such that $e \subset K:= \{z \in \sP^1\setminus \{\zeta\}: v(z)=-\infty\}$. Without loss of generality we can assume that $v \leq 0$ on $E$. Let $u \in \mathcal{L}$ be such that $u \leq 0$ on $E$ and let $\varepsilon >0$ be arbitrary. Then $u+\varepsilon v \leq Q_{E\setminus  e}$.   Letting $\varepsilon \to 0$ we get that $Q_E \leq Q_{E\setminus  e}$ (and hence  $Q_E = Q_{E\setminus  e}$) outside the set $K$, which has zero capacity. Hence also $G(\cdot,\zeta,E) \leq Q_E \leq Q_E ^*$ outside $K$. Let now $z \in K$. By Remark 7.38 in \cite{BR10}, we have 
 \begin{multline*}
 G(z,\zeta,E)=\limsup_{\sH^1 \ni y \to z}G(y,\zeta,E) \\
 =\limsup_{\sH^1 \ni y \to z}Q_E(y) \leq \limsup_{x \to z} Q_E(x)=Q_E^*(z),
 \end{multline*}
 which completes the proof.
 \end{proof} 

\begin{coro}  $Q_E=R_E :=\sup\{u \in \mathcal{L}', \ u|_E \leq 0\})$, where  
 \begin{multline*} \mathcal{L}' :=\bigl\{u: u\text{ is subharmonic in }\sP^1\setminus\{\zeta\},  \\ \forall a \ne \zeta \   u -\log [z,a]_\zeta \text{ extends to a function harmonic in } \sP^1\setminus\{a\} \bigr\}. 
\end{multline*}
\end{coro}

\begin{proof} With the notation as in the proof of Theorem \ref{theorem: equality}, \[
Q_E =\sup_n G_n \leq R_E.\] The other inequality is obvious.
\end{proof} 

\subsection{Comparison with the classical case} \label{sec: classical}

Fix a $\zeta \in  \sP^1$ and let $u \in \mathcal{L}=\mathcal{L}(\zeta)$. For an arbitrary  $a \ne \zeta$, rewrite 
\[
u(z) -\log [z,a]_\zeta =u(z)-\log\biggl ( \frac{[z,a]_g}{[z,\zeta]_g [a,\zeta]_g}\biggr )
\]

As $z \to \zeta$, we see (by continuity of the Hsia kernel $[\cdot,a]_g$ near $\zeta \neq a$) that $\limsup_{z \to \zeta} u(z) -\log [z,a]_\zeta= \limsup_{z \to \zeta} u(z) + \log[z,\zeta]_g$. The limit superior exists, since $u(z) -\log [z,a]_\zeta$ extends to  a function which is  subharmonic, in particular upper semicontinuous, in a  neighborhood of $\zeta$. \\

Recall that  the spherical kernel $[x,y]_g$ is an extension of the chordal metric from  $\mathbb{P}^1(K) \times \mathbb{P}^1(K)$ to  $\sP^1 \times \sP^1$. Consider now the field $\mathbb{C}$ of complex numbers with the standard (archimedean) absolute value and the point $\zeta =\infty \in  \mathbb{P}^1(\mathbb{C})$. The chordal distance $[z,\infty]$ equals $\frac{1}{\sqrt{1+|z|^2}}$ for $z \in \mathbb{C}$. The class of all functions $u$ subharmonic in $\mathbb{C}$ and such that $u(z) \leq \frac{1}{2}\log (1+|z|^2) + C_u$ with a constant $C_u$ dependent only on $u$ is the Lelong class $\mathcal{L}_\mathbb{C}$ on $\mathbb{C}$ (an analogous class of plurisubharmonic functions can be defined on $\mathbb{C}^n$ for $n >1$). Here we use the natural logarithm. Each difference $u(z)- \frac{1}{2}\log (1+|z|^2)$ extends to an $\omega$-subharmonic function $v$ on $\mathbb{P}^1(\mathbb{C})$ (where $\omega$ is the Fubini-Study form) by taking $v(\infty)=\limsup_{z \to \infty} u(z)- \frac{1}{2}\log (1+|z|^2)$, and the extensions yield a natural 1-to-1 correspondence between the Lelong class and the class of $\omega$-subharmonic functions. This  viewpoint on Lelong classes was introduced in \cite{GZ05} and it helped launch systematic (and successful) study of pluripotential theory on compact (complex) K\"ahler manifolds.  Our class $\mathcal{L}$ is a non-archimedean analog of the Lelong class, and similarly the class $\mathcal{L}'$ is an analog of the class $\mathcal{L}_\mathbb{C}^+=\{u \in \mathcal{L}_\mathbb{C}: u(z)- \frac{1}{2}\log (1+|z|^2)= \mathcal{O}(1) \mbox{ as } z \to \infty\}$. \\

Some extremal functions associated with  compact subsets of $\mathbb{C}$ are well known because of their usefulness in approximation theory. In \cite{Leja34}, F. Leja introduced the following extremal function:\\

Let $E \subset \mathbb{C}$ be a compact set. For an array $a^{(n)}=(a_0^{(n)},...,a_n^{(n)}), \ n=2,3,...$ of points in $\partial E$ (with all $a_i^{(n)}, i=0,...,n$ pairwise distinct) define the quantities  
\[ M(a^{(n)}):=\prod_{0\leq j <k \leq n} \left|a_j^{(n)} -a_k^{(n)}\right|. 
\]
Consider an array of Fekete extremal points in $\partial E$, that is, an array $b^{(n)}=(b_0^{(n)},...,b_n^{(n)}), \ n=1,2,...$  such that
 $M(b^{(n)})=\max_{a^{(n)}}M(a^{(n)})$ for every $n \geq 2$. In each row $b^{(n)}$ order the points so that $|\Delta_0(b_0,...,b_n)| \leq |\Delta_j(b_0,...,b_n)|, \ j =1,2,...n$, where \[
 \Delta_j(b_0,...,b_n)=(b_j-b_0)...(b_j-b_{j-1})(b_j-b_{j+1})...(b_j-b_n), \ j=0,1,...n.\]
  Let $L_n$ be the Lagrange interpolating polynomial $L_n(z)=\frac{(z-b_1)...(z-b_n)}{(b_0-b_1)...(b_0-b_n)}$. The function 
  \[
 L(z):=\lim_{n \to \infty} \frac{1}{n}\log |L_n(z)|
 \] 
 is well defined for all $z \in \mathbb{C}$. 

Under the assumption that $E$ is a union of (non-degenerate) continua, Leja also proved that $L$ equals the Green function $G(\cdot,\infty,E)$ for $E$ with pole at infinity. In \cite{Gorski48}, J. G\'orski (a student of Leja) proved the equality  $L(\cdot)=G(\cdot,\infty,E)$ for an arbitrary compact $E$ with positive capacity. In \cite{Siciak62}, J. Siciak (another student of Leja) defined an analog of the Leja extremal function for compact subsets of $\mathbb{C}^N$. Another extremal function for a compact subset of $\mathbb{C}^N$, $N \geq 1$, was defined in \cite{Za77} as 
\[
V_E(z):=\sup\{ u \in \mathcal{L}, u|_E \leq 0 \},
\]
where (for $N >2$) $\mathcal{L}$ is the class of all plurisubharmonic functions $u$ with logarithmic growth, $u(z) \leq \frac{1}{2}\log (1+\|z\|^2) + C_u$ (the Lelong class). Proofs that $V_E=\log L(\cdot,E)$ (different ones) were given in \cite{Za77}  for $E$ with $V_E$ continuous, and in \cite{Siciak81} (Theorem 4.12)   for an arbitrary non-polar compact $E$ (see also \cite{Klimek91}, Theorem 5.1.7 and, in an even more general setting of algebraic varieties embedded in $\mathbb{C}^N$, \cite{Ze91}, Th\'eor\`eme 5.1). It was observed without proof in both \cite{Za77} (beginning of Section 4) and \cite{Siciak81} that $V_E(\cdot)=G(\cdot,\infty,E)$
when $n=1$.  \\

Note that using our method of  proof of Theorem \ref{theorem: equality}
 one can prove directly that $V_E^*(\cdot)=G(\cdot,\infty,E)$ in $\mathbb{C}$, without relying on the separate equalities of each of these functions with the function $\log L(\cdot,E)$. As before, we consider $G(\cdot,\infty,E)$ as the Robin constant of the set $E$ minus the equlibrium potential. We can also assume without loss of generality that $\mathbb{C} \setminus E$ is connected.  Recall  that in proving the equality $G_E=Q_E^*$ in the non-archimedean case we took advantage of the possibility of exhausting  a domain with a sequence of subdomains whose boundaries do not contain sets of zero capacity (simple domains). There are no such subdomains in $\mathbb{C}$  with the standard absolute value, where each point has zero logarithmic capacity. Hence different auxiliary results must be used. First, for a general unbounded domain $D \subset \mathbb{C}$ we can take (as in \cite{Gorski48}) an exhaustion of $D$ by domains  $D_n, \ n \geq 1$, such that, for every $n$, $\overline{D_n} \subset D_{n+1}$, $G(\cdot,\infty,\mathbb{C}\setminus D_n)|_{\mathbb{C}\setminus D_n}=0$ and $D_n \to D$ in the sense of Carath\'eodory convergence of domains. Then $G(\cdot,\infty,\mathbb{C}\setminus D)=\lim_{n\to \infty}G(\cdot,\infty,\mathbb{C}\setminus D_n)=\lim_{n \to \infty} V_{\mathbb{C}\setminus D_n}$. Second, we need a way to compare $V_{E\cup F}^*$  with $V_E^*$ for every $E$ bounded and $F$ of zero capacity. In fact, the equality $V_{E\cup F}^*=V_E^*$ holds in $\mathbb{C}^N, \ N \geq 1$: see Proposition 3.11 in \cite{Siciak81} or Corollary 5.2.5 in \cite{Klimek91}. With these tools in place, the arguments of Theorem \ref{theorem: equality} go through in the archimedean case.\\

Finally, let us note that other characterizations of the Green function as the extremal function relative to a class of  functions with certain growth are also available in the classical theory. For example, see \cite{Doob01}, formula 1.XIII (18.1) in the logarithmic potential case; see further  \cite{Doob01}, Theorem 1.VII.2 in the Newtonian potential case in $\mathbb{R}^n$, $n >2$.

\def\cprime{$'$}


\begin{thebibliography}{10}

\bibitem{AGS}
{\sc Ambrosio,~L., Gigli,~N. {\rm\ and } Savar\'e,~G.} . {\em Gradient Flows in Metric Spaces and in the Space of Probability Measures.} Basel: ETH Z\"urich, Birkh\"auser Verlag (2005)

\bibitem{AG}
{\sc Armitage,~D.{\rm\ and }Gardiner,~S.} 
{\em Classical potential theory.} 
Springer Monographs in Mathematics. Springer-Verlag London, Ltd., London, 2001. xvi+333 pp. ISBN: 1-85233-618-8 
%MR1801253 

%\bibitem{Au}
%{\sc Autissier,~P.} {\em Autour du th\'eor\`eme de
%Fekete-Szeg\H o}, Chapter VIII in \cite{PR21}, 329-340


\bibitem{BR10}
{\sc Baker,~M.{\rm\ and }Rumely,~R.} {\em Potential theory and dynamics on the
  {B}erkovich projective line}, Vol. 159 of {\em Mathematical Surveys and
  Monographs}, American Mathematical Society, Providence, RI (2010).

%\bibitem{BenedettoBook}
%{\sc Benedetto,~R.~L.} {\em Dynamics in one non-archimedean variable}, Vol. 198
  %of {\em Graduate Studies in Mathematics}, American Mathematical Society,
 % Providence, RI (2019).

%\bibitem{Berkovichbook}
%{\sc Berkovich,~V.~G.} {\em Spectral theory and analytic geometry over
%  non-{A}rchimedean fields}, Vol.~33 of {\em Mathematical Surveys and
 % Monographs}, American Mathematical Society, Providence, RI (1990).
 
 \bibitem{Brelot38}
{\sc Brelot,~M.}
{\em Sur le potentiel et les suites de fonctions sous-harmoniques.} 
C. R. Acad. Sci., Paris 207 (1938), 836-838.
 
 \bibitem{Cartan45} {\sc  Cartan, ~H. }  {\em Th\'eorie du potentiel newtonien: \'energie, capacit\'e, suites de potentiels.} Bull. Soc. Math. France 73 (1945), 74-106. 
 %MR0015622
  
  \bibitem{Doob01}
{\sc Doob,~J.~L.} {\em Classical potential theory and its probabilistic counterpart.} Reprint of the 1984 edition. Classics in Mathematics. Springer-Verlag, Berlin, 2001. xxvi+846 pp. ISBN: 3-540-41206-9
%MR1814344 

\bibitem{DFN15} {\em Berkovich spaces and applications.}
Edited by Antoine Ducros, Charles Favre and Johannes Nicaise. Lecture Notes in Mathematics, 2119. Springer, Cham, 2015. xx+413 pp. ISBN: 978-3-319-11028-8; 978-3-319-11029-5



\bibitem{Favre11}
{\sc Favre,~C.} Countability properties of some Berkovich spaces, pp. 119-132 in \cite{DFN15}. 

%MR3330764

\bibitem{FJbook}
{\sc Favre,~C.{\rm\ and }Jonsson,~M.} {\em The valuative tree},  
   Lecture Notes in Mathematics, Vol. 1853, Springer-Verlag, Berlin (2004).


 
 \bibitem{Gorski48} 
 {\sc G\'orski,~J.} 
  {\em Sur l'\'equivalence de deux constructions de la fonction de Green g\'en\'eralis\'ee d'un domaine plan quelconque.}  Ann. Soc. Polon. Math. 21, (1948). 70-73. 
%MR0027382


\bibitem{GZ05} 
{\sc Guedj,~V. {\rm\ and } Zeriahi,~A.} Intrinsic capacities on compact K\"{a}hler manifolds,
\emph{J. Geom. Anal.} 15 (2005), no. 4, 607- 639

  
  \bibitem{HLP14}
 {\sc Hrushovski, ~E.; Loeser, ~F.; Poonen, ~B.} Berkovich spaces embed in Euclidean spaces. \emph{Enseign. Math.} 60 (2014), no. 3-4, 273-292.

\bibitem{Jonsson15}
{\sc Jonsson,~M.} Dynamics on Berkovich spaces in low dimensions, pp.  205--366 in \cite{DFN15}.


%Ka̧kol, J. (PL-POZNM); Kubzdela, A. (PL-POZNT-CE)
%On non-Archimedean quantitative compactness theorems. (English summary) Advances in non-Archimedean analysis, 91–102,
%MR3509800 Contemp. Math., 665, Amer. Math. Soc., Providence, RI, 2016
  
  \bibitem{Klimek91}
{\sc Klimek,~M.} {\em Pluripotential theory.} London Mathematical Society Monographs. New Series, 6. Oxford Science Publications. The Clarendon Press, Oxford University Press, New York, 1991. xiv+266 pp. ISBN: 0-19-853568-6
%MR1150978 

\bibitem{Leja34} {\sc Leja,~F.}
{\em Sur les suites de polyn\^omes, les ensembles fermes et la fonction de Green.}   
Ann. Soc. Polon. Math. 12, 57-71 (1934). 
%Zbl 0010.20103

\bibitem{Leja45} {\sc Leja,~F.}  {\em Sur les suites de polyn\^omes et la fonction de Green.} I.  Ann. Soc. Polon. Math. 18, (1945). 4-11.
%MR0018192

\bibitem{LG86} {\sc Lelong,~P.; Gruman,~L.} {\em  Entire functions of several complex variables.} Springer-Verlag Heidelberg 1986
  
  \bibitem{LW65}
 {\sc Loeb, ~P.; Walsh, ~B.} The equivalence of Harnack's principle and Harnack's inequality in the axiomatic system of Brelot. \emph{Ann. Inst. Fourier (Grenoble)} 15 (1965), fasc. 2, 597-600. 
  
 
 \bibitem{MO12}
 {\sc Moret-Bailly,~L.} When is a valued field second-countable? MathOverflow question 96999, May 15, 2012, answered by user Moshe (10174),  \url{https://mathoverflow.net/questions/96999/when-is-a-valued-field-second-countable}

%\bibitem{Oku2020}
%{\sc Okuyama,~Y.} Uniform perfectness of the Berkovich Julia sets in non-archimedean dynamics, preprint,  arxiv.org/abs/2011.00314


 
 
%\bibitem{Payne15}
%{\sc Payne,~S.} Topology of nonarchimedean analytic spaces and relations to complex algebraic geometry. {\em Bull. Amer. Math. Soc. (N.S.)} 52 (2015), no. 2, 223-247
% MR3312632 
 
% \bibitem{PR21}
%{\sc Peyre,~E.;  R\'emond,~G.} (eds.), {\em Arakelov Geometry and Diophantine Applications},
%Lecture Notes in Mathematics 2276, Springer 2021

%\bibitem{Ransford95}
%{\sc Ransford,~T.} {\em Potential theory in the complex plane.} London Mathematical Society Student Texts, 28. Cambridge University Press, Cambridge, 1995. x+232 pp. ISBN: 0-521-46120-0; 0-521-46654-7


\bibitem{Rumely89}  {\sc Rumely,~R. ~S.}  {\em Capacity theory on algebraic curves.} Lecture Notes in Mathematics, 1378. Springer-Verlag, Berlin, 1989. iv+437 pp. ISBN: 3-540-51410-4 
 %MR1009368 
 
 
 % \bibitem{Siciak64}
%{\sc Siciak,~J.}

\bibitem{Siciak62}
{\sc Siciak,~J.} \emph{On some extremal functions and their applications in the theory of analytic functions of several complex variables}, \emph{Trans. Amer. Math. Soc.} 105 (1962), 322-357
  
\bibitem{Siciak81}
{\sc Siciak,~J.} Extremal plurisubharmonic functions in $\mathbb{C}^n$. \emph{Ann. Polon. Math.} 39 (1981), 175-211.

%\bibitem{Siciak97a}
%{\sc Siciak,~J.} Wiener's type regularity criteria on the complex plane,
%  \emph{Ann. Polon. Math.} ~66 (1997),  203--221, Volume dedicated to the memory of W\l odzimierz
%  Mlak.
  
  
%\bibitem{Siciak97b}
%{\sc Siciak,~J.}Wiener's type sufficient conditions in $\mathbb{C}^N$.\emph{ Univ. Iagel. Acta Math.}  35 (1997), 47-74.


 
 

\bibitem{ThuillierThesis}
{\sc Thuillier,~A.} {\em Th{\'e}orie du potentiel sur les courbes en
  g{\'e}om{\'e}trie analytique non archim{\'e}dienne. Applications {\`a} la
  th{\'e}orie d'Arakelov}, PhD thesis, Universit{\'e} Rennes 1 (2005).
  
  \bibitem{Wanner2016}
{\sc Wanner,~V.} Harmonic Functions on the Berkovich
Projective Line, Master's Thesis in Arithmetic Geometry, University of Regensburg, 
Department of Mathematics, April 28, 2016 
\url{https://d-nb.info/1136471502/34}
  
\bibitem{Za77} 
  {\sc Zaharjuta,~V. ~P.}  
  Extremal plurisubharmonic functions, orthogonal polynomials, and the  Bern\v ste\u {\i}n-Walsh theorem for functions of several complex variables. (Russian) {\em Ann. Polon. Math.} 33 (1976/77), no. 1-2, 137-148.
 %MR0444988 
 
 
  \bibitem{Ze91} 
  {\sc Zeriahi,~A.} Fonction de Green pluricomplexe \`{a}
p\^{o}le  \`{a} l'infini sur un espace de Stein parabolique et
applications, \emph{Math. Scand. } 69 (1991), 89-126
  
% \url{ https://math.stackexchange.com/questions/608799/arzela-ascoli-net-question}
 
% \url{https://math.stackexchange.com/questions/2744950/relationship-between-equicontinuity-and-total-boundedness}
 
 %L. Helms: Theorem 1.2.4 (Arzela-Ascoli for nets)
 
%\url{ https://en.wikipedia.org/wiki/Net_(mathematics)}

%\url{https://math.stackexchange.com/questions/2318132/bolzano-weierstrass-theorem-for-nets}

\end{thebibliography}
\end{document}